\documentclass[a4paper,11pt,reqno, english]{amsart}  
\usepackage[utf8]{inputenc}
\usepackage[T1]{fontenc}
\usepackage{amsmath,amsthm}
\usepackage{amsfonts,amssymb,enumerate}
\usepackage{url,paralist}
\usepackage{mathtools}  
\usepackage[colorlinks=true,urlcolor=blue,linkcolor=red,citecolor=magenta]{hyperref}
\usepackage{enumerate}
\usepackage{tikz}
\usepackage[left=1in,right=1in,top=1in,bottom=1in]{geometry}

\theoremstyle{plain}
\newtheorem{thm}{Theorem}[section]
\newtheorem{lem}[thm]{Lemma}
\newtheorem{cor}[thm]{Corollary}

\newtheorem*{thm*}{Theorem}

\theoremstyle{definition}

\newtheorem{rem}[thm]{Remark}

\newcommand{\set}[1]{\left\{ #1\right\}}

\newcommand{\Z}{\mathbb{Z}}
\newcommand{\R}{\mathbb{R}}
\renewcommand{\emptyset}{\varnothing}

\DeclareMathOperator{\conv}{\mathrm{conv}}

\DeclareMathOperator{\diam}{\mathrm{diam}}

\begin{document}

\title[Colorful Borsuk--Ulam theorems]{Colorful Borsuk--Ulam theorems and applications}


\author[F. Frick \and Z. Wellner]{Florian Frick \and Zoe Wellner}

\address[FF, ZW]{Dept. Math. Sciences, Carnegie Mellon University, Pittsburgh, PA 15213, USA}
\email{\{ffrick, zwellner\}@andrew.cmu.edu}

\thanks{FF and ZW were supported by NSF CAREER Grant DMS 2042428.}

\date{September 25, 2023}
\maketitle



\begin{abstract}
\small 
We prove a colorful generalization of the Borsuk--Ulam theorem and derive colorful consequences from it, such as a colorful generalization of the ham sandwich theorem. Even in the uncolored case this specializes to a strengthening of the ham sandwich theorem, which given an additional condition, contains a result of B\'{a}r\'{a}ny, Hubard, and Jer\'{o}nimo on well-separated measures as a special case. We prove a colorful generalization of Fan's antipodal sphere covering theorem, we derive a short proof of Gale's colorful KKM theorem, and we prove a colorful generalization of Brouwer's fixed point theorem.
Our results also provide an alternative between Radon-type intersection results and KKM-type covering results. Finally, we prove colorful Borsuk--Ulam theorems for higher symmetry.
\end{abstract}

\section{Introduction}

\emph{Colorful} (or \emph{rainbow}) results are popular across combinatorics and discrete geometry. These results take the following general form: If sets $S_1, \dots, S_n$ have property~$P$, then there is a transversal with property~$P$. Here a \emph{transversal} of $S_1,\dots, S_n$ is a set $\{s_1,\dots, s_n\}$ with $s_i \in S_i$. For example, B\'ar\'any's colorful Carath\'eodory theorem \cite{Barany1981} states that if $S_1,\dots, S_{d+1} \subset \R^d$ satisfy $0 \in \conv S_i$ for all~$i \in [d+1] = \{1,\dots, d+1\}$, then there is a transversal $S$ with $0 \in \conv S$. The ``non-colorful'' case $S_1 = \dots = S_{d+1}$ reduces to Carath\'eodory's theorem \cite{Caratheodory1911} that if $0$ is in the convex hull of $S \subset \R^d$, then some convex combination of at most $d+1$ elements of $S$ is equal to~$0$.

Other colorful results in geometry include colorful Helly theorems~\cite{Barany1981, kalai2005}, Gale's colorful generalization of the KKM theorem \cite{Gale1984}, and colorful versions of Tverberg's theorem~\cite{barany1989, zivaljevic1992, blagojevic2015}. Prominent examples in combinatorics include results on rainbow arithmetic progressions~\cite{jungic2003, axenovich2004}, rainbow matching results (such as~\cite{AharoniBerger2009}) and rainbow Ramsey results (see~\cite{fujita2010} for a survey) among several others.

Topological methods have proven to be a powerful tool in attacking combinatorial and discrete-geometric problems~\cite{bjorner1995, deLongueville2013, kozlov2008, Matousek2003book}. Among the standard techniques are fixed point theorems (an early example is Nash's result on equilibria in non-cooperative games~\cite{Nash51})
and equivariant methods such as the Borsuk--Ulam theorem (see for example~\cite{Matousek2003book, blagojevic2017}), which states that any continuous map $f\colon S^d \to \R^d$ from the $d$-sphere~$S^d$, which is \emph{odd} (i.e., $f(-x) = -f(x)$ for all~$x$), has a zero. 

Here we ask: \emph{Can colorful results be lifted to colorful topological methods?} For fixed point theorems this is true in quite some generality~\cite{shih1993, FrickHoustonEdwardsMeunier2017, FrickZerbib2019}. There is an abundance of generalizations of the Borsuk--Ulam theorem; see for example~\cite{Fan1952, MeunierSu2019, Dold1983, Bourgin1955, yang1954}.
Here we prove a colorful generalization of the Borsuk--Ulam theorem. We introduce one piece of terminology: We say that a matrix $A\in \R^{d \times d}$ has \emph{rows in intersecting cube facets} if for any two distinct rows $a$ and $b$ of $A$ whenever there is a $j \in [d]$ such that $|a_j| \ge |a_i|$ and $|b_j| \ge |b_i|$ for all $i \in [d]$, then $a_jb_j > 0$. That is, if the largest entries in absolute value in two different rows occur in the same column, then they cannot have opposite sign. If we normalize all rows to have sup-norm~$1$, that is, to lie on the boundary of the hypercube in~$\R^d$, then $A$ has rows in intersecting cube facets if no two rows lie strictly within vertex-disjoint facets of the hypercube. We can now state our colorful Borsuk--Ulam theorem:

\begin{thm}
\label{thm:colorfulBU}
   Let $f \colon S^d \to \R^{(d+1) \times (d+1)}$ be an odd and continuous map. Then there is an ${x\in S^d}$ such that either $f(x)$ does not have rows in intersecting cube facets or there is a permutation $\pi$ of $[d+1]$ such that $f(x)_{\pi(i)i}$ is non-negative and $|f(x)_{\pi(i)i}| \ge |f(x)_{\pi(i)j}|$ for all~$i, j \in [d+1]$.
\end{thm}

Thus for any odd map from the $d$-sphere to $\R^{(d+1) \times (d+1)}$ one matrix $A$ in the image either has rows that are ``almost opposite'' or the maximal absolute values of entries of $A$ in each row form a row-column transversal and these entries are non-negative. Here we think of the rows of $f(x)$ as $d+1$ odd maps $S^d \to \R^{d+1}$. Theorem~\ref{thm:colorfulBU} is indeed a colorful generalization of the Borsuk--Ulam theorem: Let $f \colon S^d \to \R^d$ be an odd and continuous map. Let $\widehat f \colon S^d \to \R^{d+1}$ be the map obtained from $f$ by appending a zero in the last coordinate. Define $F\colon S^d \to \R^{(d+1) \times (d+1)}$ to have all rows equal to~$\widehat f$. By Theorem~\ref{thm:colorfulBU} there is an $x\in S^d$ such that either for some $j \in [d+1]$ we have that $|\widehat f(x)_j| \ge |\widehat f(x)_\ell|$ for all $\ell \in [d+1]$ and $\widehat f(x)_j$ is both non-negative and non-positive, which implies $\widehat f(x) = 0$, or there is a permutation $\pi$ of $[d+1]$ such that $F(x)_{\pi(i) i}$ is largest in absolute value in row $\pi(i)$ for all~$i$. Since all rows of $F$ are equal to $\widehat f$ this means that each entry $\widehat f(x)_j$ is maximal in absolute value among the coordinates of~$\widehat f(x)$, so these absolute values are all equal. Since $\widehat f(x)_{d+1} = 0$, we have that $\widehat f(x) = 0$.

Here we collect consequences and generalizations of Theorem~\ref{thm:colorfulBU}:

\begin{compactenum}[(a)]
    \item The classical Ham Sandwich theorem, conjectured by Steinhaus and proved by Banach (see~\cite{beyer2004, StoneTukey1942}), asserts that for any $d$ probability measures on $\R^d$ with continuous density, there is an affine hyperplane~$H$ that bisects all probability measures, that is, both halfspaces of $H$ have measure~$\frac12$ in all $d$ probability measures. We prove a colorful generalization of the ham sandwich theorem; see Theorem~\ref{colHS}. When we specialize this to the non-colorful case, we derive a strengthening of the ham sandwich theorem (Theorem~\ref{kfhs}), which contains a close variant of a measure partition result due to B\'{a}r\'{a}ny, Hubard, Jer\'{o}nimo~\cite{BaranyHubardJeronimo2008} on well-separated measures as a special case; see Corollary~\ref{cor:bhj}.
    \item Fan proved that if $A_1, \dots, A_{d+1} \subseteq S^d$ are closed such that $A_i \cap (-A_i) = \emptyset$ for all $i \in [d+1]$ and $S^d = \bigcup_i A_i \cup(-A_i)$, then $A_1 \cap (-A_2) \cap A_3 \cap \dots \cap (-1)^d A_{d+1} \ne \emptyset$. We establish a colorful generalization; see Theorem~\ref{colkf}.
    \item The Borsuk--Ulam theorem easily implies Brouwer's fixed point theorem, which asserts that any continuous self-map of a closed $d$-ball has a fixed point. 
    Similarly, the colorful Borsuk--Ulam theorem implies a colorful Brouwer's fixed point theorem; see Theorem~\ref{thm:colBrouwer}. We give a simple proof of Gale's colorful KKM theorem~\cite{Gale1984} as a consequence of our main result; see Corollary~\ref{cor:colKKM}. Again, the uncolored version is stronger than the classical KKM theorem about set coverings of the $d$-simplex~$\Delta_d$. In this case, we obtain an alternative between KKM results and the topological Radon theorem, which unifies both of these results; see Theorem~\ref{thm:RadonKKM}.
    \item We give a generalization of Fan's result to $\Z/p$-symmetry for $p$ a prime; see Theorem~\ref{thm:p-fold-kyfan}. We derive a generalization of the Bourgin--Yang $\Z/p$-Borsuk--Ulam theorem (Corollary~\ref{cor:kyfan-dold}) and the corresponding set covering variant (Theorem~\ref{thm:p-cover}). We then prove a colorful generalization of Theorem~\ref{thm:p-cover}, which in the special case $p=2$ gives our earlier colorful generalization of Fan's set covering result; see Theorem~\ref{thm:col-p-kyfan}.
\end{compactenum}

Meunier and Su already proved a colorful generalization of Fan's theorem \cite{MeunierSu2019}, which also exhibits a colorful Borsuk--Ulam phenomenon. Their generalization is different from ours and neither easily implies the other. We will discuss the differences after stating our colorful generalization of Fan's theorem (Theorem~\ref{colkf}).

\section{Preliminaries}

In this section we collect a few definitions used throughout. We refer to Matou\v sek's book~\cite{Matousek2003book} for further details.
A \emph{simplicial complex} $\Sigma$ is a collection of finite sets closed under taking subsets. We refer to the ground set $\bigcup_{\sigma \in \Sigma} \sigma$ as the \emph{vertex set} of~$\Sigma$. Elements $\sigma \in \Sigma$ are called \emph{faces}; inclusion-maximal faces are \emph{facets} and two-element faces are \emph{edges}. For a simplicial complex $\Sigma$ on vertex set~$V$ we denote its \emph{geometric realization} by $|\Sigma| = \bigcup_{\sigma \in \Sigma} \conv\{e_v \ : \ v\in \sigma\} \subseteq \R^V$, where $e_v$ denote the standard basis vectors of~$\R^V$ and $\conv(X)$ denotes the convex hull of the set~$X$. The simplicial complex of all subsets of $[n] = \{1,2,\dots,n\}$ is the \emph{$(n-1)$-simplex}~$\Delta_{n-1}$. For ease of notation we will denote the geometric realization of $\Delta_{n-1}$ also by~$\Delta_{n-1}$. Observe that $\Delta_{n-1}$ is the convex hull of the standard basis vectors in~$\R^n$. We call $\Sigma$ a \emph{triangulation} of $S^d$ if $\Sigma$ is a simplicial complex whose geometric realization is homeomorphic to~$S^d$.

Let $\Sigma$ and $\Sigma'$ be simplicial complexes on vertex sets $V$ and~$V'$, respectively. A map $\varphi\colon V \to V'$ is \emph{simplicial} if for every $\sigma \in \Sigma$ we have that $\varphi(\sigma) \in \Sigma'$. In this case we write $\varphi \colon \Sigma \to \Sigma'$. By convex interpolation any simplicial map induces a continuous map $|\Sigma| \to |\Sigma'|$.

A triangulation $\Sigma$ of $S^d$ is \emph{centrally symmetric} if there is a simplicial map $\iota \colon \Sigma \to \Sigma$ and homeomorphism $h \colon |\Sigma| \to S^d$ such that for every $x\in |\Sigma|$ we have that $h(\iota(x)) = -h(x)$. The smallest centrally symmetric triangulation of $S^d$ is given by the (boundary of the) \emph{crosspolytope}~$\partial\Diamond_{d+1}$, the simplicial complex on $V = \{1, \dots, d\} \cup \{-1, \dots, -d\}$, where $\sigma \subseteq V$ is a face of~$\partial\Diamond_{d+1}$ if for every $i \in \sigma$ we have that $-i \notin \sigma$. The geometric realization $|\partial\Diamond_{d+1}|$ can be realized in~$\R^{d+1}$ with convex faces by taking the boundary of the convex hull of~$\{\pm e_1, \dots, \pm e_{d+1}\}$. 

A simplicial map $\iota \colon \Sigma \to \Sigma$ induces a \emph{$\Z/p$-action} if the $p$-fold composition $\iota^p$ is the identity. In this case $\Sigma$ is a \emph{$\Z/p$-equivariant triangulation}. For $s \in \Z/p$ and $x \in |\Sigma|$ we write $s\cdot x$ for~$\iota^s(x)$. The $\Z/p$-action is \emph{free} if $s\cdot x \ne x$ for all $s \in \Z/p \setminus \{0\}$ and all $x \in |\Sigma|$. Given two spaces $X$ and $Y$ (homeomorphic to simplicial complexes) with $\Z/p$-actions, a map $f\colon X \to Y$ is \emph{$\Z/p$-equivariant} if $f(s\cdot x) = s\cdot f(x)$ for all $s\in \Z/p$ and all $x\in X$.

A map $f\colon S^d \to \R^n$ is \emph{antipodal} or \emph{odd} if $f(-x) = -f(x)$ for all $x \in S^d$. We reserve the term \emph{map} for a continuous function. The same definition applies for a map $f\colon S^d \to \R^{n \times m}$ to the space of $(n \times m)$-matrices. For any such map, we write $f_i \colon S^d \to \R^m$ for the map to the $i$th row of~$f$ and $f_{ij} \colon S^d \to \R$ for the map to the entry in row~$i$ and column~$j$ of~$f$. The \emph{degree} $\deg f$ of a map $f \colon S^d \to S^d$ is the integer $k$ such that the induced map on top homology $f_*\colon H_d(S^d) \cong \Z \to H_d(S^d) \cong \Z$ is multiplication by~$k$. The Borsuk--Ulam theorem~\cite{Borsuk1933} can be stated in various forms; here we collect three such statements:

\begin{thm}[Borsuk--Ulam theorem]
\label{thm:BU}
    ~
    \begin{compactenum}[(a)]
        \item Any odd map $f\colon S^d \to \R^d$ has a zero. 
        \item Any odd map $f\colon S^d \to S^d$ has odd degree.
        \item For any odd map $f\colon S^d \to \R^{d+1}$ there is an $x \in S^d$ such that all coordinates of~$f(x)$ are the same.
    \end{compactenum}
\end{thm}

Item~(b) implies item~(a), which is easily seen to be equivalent to the statement that any odd map $S^{d-1}\to S^{d-1}$ has nonzero degree. For item~(c) observe that for the diagonal $D = \{x\in \R^{d+1} \ : \ x_1 = x_2 = \dots = x_{d+1}\}$ the composition of $f \colon S^d \to \R^{d+1}$ with the projection $\R^{d+1} \to \R^{d+1} / D \cong \R^d$ is an odd map. This composition has a zero if and only if there is an $x \in S^d$ with $f(x) \in D$. We will use one immediate corollary of the Borsuk--Ulam theorem, which we state below. Any non-surjective map $S^d \to S^d$ has degree zero; thus we get:

\begin{cor}
\label{cor:surj}
    Any odd map $f\colon S^d \to S^d$ is surjective.
\end{cor}

The Borsuk--Ulam theorem has been generalized to $G$-equivariant maps beyond $G = \Z/2$; see for example Dold~\cite{Dold1983}. Here we will need the following (see~{\cite[Cor.~2.2]{kushkuley2006}}):

\begin{lem}
    \label{lem:dold}
    For any free $\Z/p$-action on~$S^d$, any $\Z/p$-equivariant map $f\colon S^d\to S^d$ has degree $1 \ \mathrm{mod} \ p$.
\end{lem}

Let $\Sigma$ be a simplicial complex on vertex set~$V$. The \emph{deleted join} $\Sigma^{*2}_\Delta$ of $\Sigma$ is a simplicial complex on vertex set $V\times \{1,2\}$, where $(\sigma_1 \times \{1\}) \cup (\sigma_2 \times \{2\})$ is a face of $\Sigma^{*2}_\Delta$ if $\sigma_1$ and $\sigma_2$ are faces of~$\Sigma$ such that $\sigma_1 \cap \sigma_2 = \emptyset$. The deleted join of the $n$-simplex is $(\Delta_n)^{*2}_\Delta = \partial\Diamond_{n+1}$ the boundary of the crosspolytope. Notice that any point $z \in |(\Delta_n)^{*2}_\Delta|$ in the geometric realization of the boundary of the crosspolytope is of the form $\lambda x + (1-\lambda) y$ for $\lambda \in [0,1]$ and $x,y \in |\Delta_n|$ that are contained in vertex-disjoint faces. This is true more generally, for points in the geometric realization of the deleted join of any simplicial complex. We surpress bars, and denote the geometric realization of the deleted join of the simplex by $(\Delta_n)^{*2}_\Delta$ for ease of notation. Thus the notation $\lambda x + (1-\lambda)y \in (\Delta_n)^{*2}_\Delta$ refers to the point in $|(\Delta_n)^{*2}_\Delta|$ determined by $\lambda \in [0,1]$ and $x,y \in |\Delta_n|$ in vertex-disjoint faces.
The \emph{$p$-fold join} $\Sigma^{*p}$ of $\Sigma$ is a simplicial complex on vertex set $V\times \{1,2,\dots, p\}$, where $(\sigma_1 \times \{1\}) \cup (\sigma_2 \times \{2\})\cup\dots\cup(\sigma_p \times \{p\})$ is a face of $\Sigma^{*p}_\Delta$ if for all~$i$, $\sigma_i$ is a face of~$\Sigma$. If we additionally require that $\sigma_i \cap \sigma_j = \emptyset$ when $i\neq j$ we get the \emph{$p$-fold deleted join}~$\Sigma^{*p}_\Delta$. The $p$-fold join of $S^d$ is homeomorphic to~$S^{p(d+1)-1}$.
The \emph{barycentric subdivision}~$\Sigma'$ of $\Sigma$ is the simplicial complex on vertex set~$\Sigma$ with $\{\sigma_1, \dots, \sigma_k\}$ is a face of~$\Sigma'$ if $\sigma_1 \subset \sigma_2 \subset \dots \subset \sigma_k$.

\begin{figure}
\centering
\begin{tikzpicture}[scale=0.6]
\tikzstyle{vert}=[circle, draw,
                       inner sep=2pt, minimum width=5pt]
\node[vert] (A) at (0,0){};
\node[vert] (B) at (3.5, -6){};
\node[vert] (C) at (-3.5, -6){};

\node (A1) at (0,1){1};
\node (B1) at (4.5, -6.5){2};
\node (C1) at (-4.5, -6.5){3};

\path[-, line width=1]
    (A) edge (B)
    (B) edge (C)
    (C) edge (A);

\end{tikzpicture}
\hspace{0.1in}
\begin{tikzpicture}[scale=0.6]
\tikzstyle{vert}=[circle, draw, inner sep=2pt, minimum width=5pt]

\node[vert] (A) at (0,0){};
\node[vert] (B) at (3.5, -6){};
\node[vert] (C) at (-3.5, -6){};
\node[vert] (D) at (0, -6){};
\node[vert] (E) at (-1.75, -3){};
\node[vert] (F) at (1.75, -3){};
\node[vert] (G) at (0, -4){};

\node (A1) at (0,1){$\{1\}$};
\node (B1) at (4.5, -6){$\{2\}$};
\node (C1) at (-4.5, -6){$\{3\}$};
\node (D1) at (0, -6.5){$\{2,3\}$};
\node (E1) at (-2.75, -3){$\{1,3\}$};
\node (F1) at (2.75, -3){$\{1,2\}$};
\node (G1) at (1.2, -4.5){$\{1,2,3\}$};

\path[-, line width=1]
    (A) edge (F)
    (F) edge (B)
    (B) edge (G)
    (G) edge (E)
    (E) edge (A)
    (A) edge (G)
    (G) edge (D)
    (D) edge (C)
    (C) edge (G)
    (G) edge (F);

\path[-, line width=1]
    (C) edge (E);

\path[-, line width=1]
    (D) edge (B);
    
\end{tikzpicture}
\caption{A filled triangle and its barycentric subdivision, which has a vertex for every face of the triangle.
}
\label{fig:barycenter}
\end{figure}
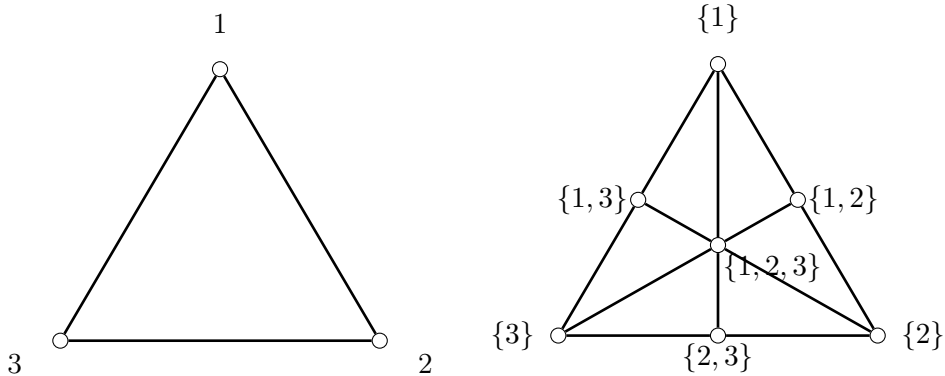

\section{Proof of the colorful Borsuk--Ulam theorem}

In 1952 Ky Fan published his ``combinatorial lemma'' generalizing Tucker's lemma \cite{Fan1952}, which is a discretized version of the Borsuk--Ulam theorem. Ky Fan's lemma applies to iterated barycentric subdivisions of the boundary of the crosspolytope and states that if the vertices of such a subdivision are labelled with the $2m$ numbers $\{\pm 1, \pm 2, \dots, \pm m\}$ such that labels of antipodal vertices sum to zero, while labels of vertices connected by an edge do not sum to zero, then the number of facets labelled $\{k_1, -k_2, k_3, \dots, (-1)^{d}k_{d+1}\}$, where $1 \le k_1 < k_2 < \dots < k_{d+1}$, is odd. Below we give a short proof of a version of Ky Fan's lemma that applies in greater generality, that is, for any centrally symmetric triangulation, while having a slightly weaker conclusion. In fact, with more care one could derive a generalization of Ky Fan's lemma from our setup below, but we will not need this generality. The proof we present here is not new; see De Loera, Goaoc, Meunier and Mustafa~\cite{LoeraGoaocMeunierMustafa2017}. Further note that generalizations of this theorem have been proven in other settings; for example, Musin proved a generalization of Fan's lemma for manifolds \cite{Musin2012} and \v{Z}ivaljevi\'{c} proved a generalization for oriented matroids \cite{Zivaljevic2010}.

\begin{thm}\label{newkf}
    Let $\Sigma$ be a centrally symmetric triangulation of~$S^d$ with vertex set~$V$. Let $\ell\colon V \to \{\pm 1, \pm 2, \dots, \pm (d+1)\}$ be a map with $\ell(-v) = -\ell(v)$ for all~$v \in V$. Fix signs $s_1, \dots, s_{d+1} \in \{-1,+1\}$. Then either $\Sigma$ has an edge $e$ with $\ell(e) = \{-j,+j\}$ for some $j \in [d+1]$ or $\Sigma$ has a facet $\sigma$ with $\ell(\sigma) = \{s_1\cdot 1, \dots, s_{d+1}\cdot (d+1)\}$.
\end{thm}

\begin{proof}
    Assume that $\Sigma$ has no edge $e$ with $\ell(e) = \{-j,+j\}$ for $j \in [d+1]$. Then the map $\ell$ induces a simplicial map $L\colon \Sigma \to \partial\Diamond_{d+1}$ to the boundary of the crosspolytope~$\partial\Diamond_{d+1}$ by identifying label $j$ with standard basis vector $e_j$, and similarly identifying $-j$ with~$-e_j$. This map is odd, and thus $L$ is surjective by Corollary~\ref{cor:surj}. In particular, there is a facet $\sigma$ that $L$ maps to the facet $\{s_1\cdot 1, \dots, s_{d+1}\cdot (d+1)\}$ in~$\partial\Diamond_{d+1}$, which finishes the proof.
\end{proof}

By taking limits, we derive a version of the theorem above for set coverings of the sphere instead of labellings of triangulations. Fan already derived such a set covering variant in his original paper. We include a proof for completeness. 

\begin{thm}
\label{thm:set-kyfan}
    Let $A_1, \dots, A_{d+1} \subseteq S^d$ be closed sets such that $S^d = \bigcup_i A_i \cup \bigcup_i (-A_i)$. Fix signs $s_1, \dots, s_{d+1} \in \{-1,+1\}$. Either there is an $i\in [d+1]$ such that $A_i \cap (-A_i) \ne \emptyset$ or $\bigcap_i s_i A_i \ne \emptyset$.
\end{thm}

\begin{proof}
    Assume that $A_i \cap (-A_i) = \emptyset$ for all~$i$. Let $\varepsilon>0$ be sufficiently small that the distance between $A_i$ and $-A_i$ is larger than~$\varepsilon$ for all~$i$. Let $T_\varepsilon$ be a centrally symmetric triangulation of~$S^d$ such that each facet has diameter less than~$\varepsilon$. This can be achieved by taking repeated barycentric subdivisions of a given centrally symmetric triangulation. 
    
    Let $\ell\colon V(T_{\varepsilon})\to \{\pm 1,\dots, \pm (d+1)\}$ be a labelling of the vertices of~$T_\varepsilon$ such that $\ell(v) = i$ only if $v \in A_i$, and such that $\ell(-v)=-\ell(v)$. Thus $\ell(v) = -i$ only if $v \in (-A_i)$. By our choice of~$\varepsilon$, the sum of labels of any edge is non-zero. By Theorem~\ref{newkf} there is a facet with labels $s_1\cdot 1, \dots, s_{d+1}\cdot (d+1)$. Let $x_\varepsilon$ be the barycenter of some such facet. As $\varepsilon$ approaches zero, by compactness of~$S^d$, the $x_\varepsilon$ have an accumulation point~$x$. Since the $A_i$ are closed, we have that $x\in \bigcap_i s_i A_i$.
\end{proof}

We will now derive a colorful set covering version of Fan's lemma by considering the barycentric subdivision of fine triangulations. We then label each vertex according to the $j$th set covering if the dimension of face it subdivides is~$j$. This idea is not new: It was originally used by Su~\cite{Su1999} to derive a colorful version of Sperner's lemma in order to establish results on rental harmony. Recently this idea was employed in~\cite{FrickZerbib2019} to prove colorful versions of set covering results for polytopes, such as a colorful KKMS and colorful Komiya's theorem. The new ingredient there was the application to set covering (instead of vertex labelling) results. The following is a colorful generalization of Fan's sphere covering result:

\begin{thm}\label{colkf}
    For $j \in [d+1]$ let $A_1^{(j)}, \dots, A_{d+1}^{(j)} \subseteq S^d$ be closed sets such that $S^d = \bigcup_i A_i^{(j)} \cup \bigcup_i (-A_i^{(j)})$ for each~$j$. Suppose $A_i^{(j)} \cap (-A_i^{(\ell)}) = \emptyset$ for all~$i$ and for all $j \ne \ell$. Fix signs $s_1, \dots, s_{d+1} \in \{-1,+1\}$. Then there is a permutation $\pi$ of $[d+1]$ such that $\bigcap_i s_{i} A_i^{(\pi(i))} \ne \emptyset$.
\end{thm}

\begin{proof}
    As before, let $T_\varepsilon$ be a triangulation, where every face has diameter at most~$\varepsilon > 0$. Here $\varepsilon$ is chosen sufficiently small so that no face intersects both $A_i^{(j)}$ and~$-A_i^{(\ell)}$ for $j \ne \ell$ and any~$i$. Let $T'_\varepsilon$ denote the barycentric subdivision of~$T_\varepsilon$. Similar to the proof of Theorem~\ref{thm:set-kyfan}, let $\ell\colon V(T'_{\varepsilon})\to \{\pm 1,\dots, \pm (d+1)\}$ be a labelling of the vertices of~$T'_\varepsilon$ such that $\ell(v) = i$ only if $v \in A_i^{(k)}$ and $v$ subdivides a $(k-1)$-dimensional face of~$T_\varepsilon$. We may assume that $\ell(-v)=-\ell(v)$. By our choice of~$\varepsilon$, the sum of labels of any edge is non-zero. By Theorem~\ref{newkf} there is a facet with labels $s_1\cdot 1, \dots, s_{d+1}\cdot (d+1)$. Let $x_\varepsilon$ be the barycenter of some such facet. As $\varepsilon$ approaches zero, by compactness of~$S^d$, the $x_\varepsilon$ have an accumulation point~$x$. For each small $\varepsilon > 0$, let $\pi_\varepsilon$ be the permutation of~$[d+1]$ with $\pi_\varepsilon(i) = j$ if the (unique) vertex of the facet of $T'_\varepsilon$ that contains $x_\varepsilon$ and subdivides a face of dimension $j-1$ has label~$s_i\cdot i$. (If $x_\varepsilon$ is in multiple facets, we choose one arbitrarily.) Since there are only finitely many permutations of~$[d+1]$, we can choose a sequence of $\delta_1, \delta_2, \dots$ with $\delta_n$ converges to $0$ and all permutations $\pi_{\delta_n}$ are the same. Call this permutation~$\pi$. Since the $A_i^{(j)}$ are closed, we have that $x\in \bigcap_i s_i A_i^{(\pi(i))}$.
\end{proof}

We compare this colorful generalization of Fan's theorem to the result of Meunier and Su~\cite{MeunierSu2019}, who proved a multilabelled Fan's theorem. We will translate their result to a set covering result by taking limits to make a direct comparison with Theorem~\ref{colkf}.

\begin{thm}[Meunier and Su~\cite{MeunierSu2019}]
\label{thm:ms}
    Let $m \ge 1$ and $d\ge 1$ be integers, and let $d_1, \dots, d_m$ be non-negative integers with $d_1 + \dots + d_m = d$. For $j \in [m]$ let $A_1^{(j)}, \dots, A_{d+1}^{(j)} \subseteq S^d$ be closed sets such that $S^d = \bigcup_i A_i^{(j)} \cup \bigcup_i (-A_i^{(j)})$ for each~$j$. Suppose $A_i^{(j)} \cap (-A_i^{(j)}) = \emptyset$ for all~$i, j$. Then there are strictly increasing maps $f_j\colon [d_j+1] \to [d+1]$ and signs $s_j \in \{-1,+1\}$ for $j \in [m]$ such that $$\bigcap_{j=1}^m \bigcap_{i=1}^{d_j+1} s_j\cdot (-1)^i A_{f_j(i)}^{(j)} \ne \emptyset.$$
\end{thm}

Think of the sets $A_i^{(j)}$ recorded in a matrix, with set $A_i^{(j)}$ in the $j$th row and $i$th column. Theorem~\ref{thm:ms} has no assumptions regarding intersections between sets in different rows (only on sets in the same row) and the conclusion gives an intersection among sets (with alternating signs) that form a row-transversal. Theorem~\ref{colkf} has no assumption regarding intersections of sets in the same row (only distinct rows) and the conclusion gives an intersection among sets (with prescribed signs) that form a row and column transversal.

\begin{proof}[Proof of Theorem~\ref{thm:colorfulBU}]
    Define the following sets:
    \[A_i^{(j)}=\set{ x\in S^d \ | \ f(x)_{ji}=|f(x)_{ji}| \ge |f(x)_{j \alpha }| \ \text{for all} \ \alpha \in [d+1]},\]
    \[-A_i^{(j)}=\set{ x\in S^d \ | \ -f(x)_{ji}=|f(x)_{ji}| \ge |f(x)_{j \alpha }| \ \text{for all} \ \alpha \in [d+1]}.\]

    Note that for each~$j$, the collection of sets $A_i^{(j)}$ will satisfy $S^d=\bigcup_i A_i^{(j)} \cup \bigcup_i (-A_i^{(j)})$ since for every $x\in S^d$ the maximal entry in absolute value in the $j$th row of $f(x)$ must be achieved somewhere. If $A_i^{(j)} \cap (-A_i^{(\ell)}) \neq \emptyset$ for some~$i$ and for some $j \ne \ell$ then $f(x)_{ji}=|f(x)_{ji}| \ge |f(x)_{j \alpha }|$ and $-f(x)_{\ell i}=|f(x)_{\ell i}| \ge |f(x)_{\ell\alpha }|$ for all $\alpha \in [d+1]$. Thus $f(x)$ does not have rows in intersecting cube facets.

    If  $A_i^{(j)} \cap (-A_i^{(\ell)}) = \emptyset$ for all~$i$ and for all $j \ne \ell$ apply Theorem~\ref{colkf} with all signs $s_i = +1$, to get that there is a permutation $\pi$ of $[d+1]$ such that $\bigcap_i A_i^{(\pi(i))} \ne \emptyset$. Then for $x \in \bigcap_i A_i^{(\pi(i))}$ we have that $f(x)_{\pi(i)i} = |f(x)_{\pi(i)i}| \ge |f(x)_{\pi(i)j}|$ for all~$i, j \in [d+1]$.
\end{proof}

\section{Radon--KKM alternative and the colorful KKM theorem}

Recall that $\Delta_d = \{x \in \R^{d+1} \ : \ \sum_i x_i = 1, \ x_i \ge 0 \ \forall i \in [d+1]\}$ denotes the \emph{$d$-simplex}. For a subset $J \subseteq [d+1]$ the set $\Delta_d^J = \{x \in \Delta_d \ : \ x_j = 0 \ \forall j \notin J\}$ is a \emph{face} of~$\Delta_d$. If $J = [d+1]\setminus\{i\}$, then we call $\Delta_d^J$ the \emph{$i$th facet} of~$\Delta_d$. For $f\colon \Delta_d \to \R^n$, a partition $J \sqcup J'$ of $[d+1]$ with $f(\Delta_d^J) \cap f(\Delta_d^{J'}) \ne \emptyset$ is a \emph{Radon partition} for~$f$. The following is a ``discretized'' variant of the Borsuk--Ulam theorem:

\begin{thm}[Topological Radon theorem -- Bajmoczy and B\'ar\'any~\cite{Bajmoczy1979}]
    Let $f \colon \Delta_d \to \R^{d-1}$ be continuous. Then $f$ has a Radon partition.
\end{thm}

It is no loss of generality to state the topological Radon theorem for maps $f\colon \Delta_d \to \Delta_{d-1}$ or $f\colon \Delta_d \to \partial\Delta_d \cong S^{d-1}$, where $\partial\Delta_d$ denotes the boundary of~$\Delta_d$. The topological Radon theorem is derived from the Borsuk--Ulam theorem, and in fact can be seen as a discretized version of it. Before exploring the colorful extension implied by our main result, we first investigate the non-colorful version, which will apply to maps $f\colon \Delta_d \to \Delta_d$. We will show that the topological Radon theorem, in a sense, is dual to the KKM theorem:

\begin{thm}[KKM theorem~\cite{KnasterKuratowskiMazurkiewicz1929}]
    Let $A_1, \dots, A_{d+1}$ be an open cover of~$\Delta_d$ such that for every $J \subseteq [d+1]$ we have that $\Delta_d^J \subseteq \bigcup_{j \in J} A_j$. Then $\bigcap A_i \ne \emptyset$.
\end{thm}

We say that a finite sequence of continuous maps $\alpha_1, \dots, \alpha_n \colon \Delta_d \to [0,1]$ is a \emph{partition of unity} if $\sum_i \alpha_i(x) = 1$ for all $x \in \Delta_d$. Note that this means that $\alpha = (\alpha_1, \dots, \alpha_n)$ is a map to the $(n-1)$-simplex~$\Delta_{n-1}$. A partition of unity $\alpha_1, \dots, \alpha_n \colon \Delta_d \to [0,1]$ is \emph{subordinate} to an open cover $A_1, \dots, A_n$ of~$\Delta_d$ if $\alpha_i(x) > 0$ implies $x \in A_i$. Any open cover $A_1, \dots, A_n$ of $\Delta_d$ has a partition of unity subordinate to it, since the simplex is locally compact and Hausdorff. Conversely, any continuous $\alpha \colon \Delta_d \to \Delta_{n-1}$ gives an open cover $A_i = \{x\in \Delta_d \ : \ \alpha_i(x) > 0\}$ of~$\Delta_d$. Having a Radon partition is a degeneracy of a map $\alpha \colon \Delta_d \to \Delta_d$, while for a cover $A_1, \dots, A_{d+1}$ having empty intersection, $\bigcap A_i = \emptyset$, is a degeneracy. Theorem~\ref{thm:RadonKKM} shows that these degeneracies are dual to one another.

\begin{thm}
\label{thm:RadonKKM}
     Let $A_1, \dots, A_{d+1}$ be an open cover of~$\Delta_d$ and let $\alpha_1, \dots, \alpha_{d+1} \colon \Delta_d \to [0,1]$ be a partition of unity subordinate to the cover $\{A_1,\dots,A_{d+1}\}$. Let $\alpha = (\alpha_1,\dots, \alpha_{d+1}) \colon \Delta_d \to \Delta_d$. Then $\alpha$ has a Radon partition or $\bigcap A_i \ne \emptyset$.
\end{thm}

\begin{proof}
    Let $F\colon (\Delta_d)^{*2}_\Delta \to \R^{d+1}, \ \lambda x + (1-\lambda)y \mapsto \lambda\alpha(x) - (1-\lambda)\alpha(y)$. Since $(\Delta_d)^{*2}_\Delta$ is homeomorphic to~$S^d$, by the Borsuk--Ulam theorem (Theorem~\ref{thm:BU}(c)), $F$ must hit the diagonal $D = \{(z_1,\dots, z_{d+1}) \in \R^{d+1} \ : \ z_1 = z_2 =\dots = z_{d+1}\}$. If $F(\lambda x + (1-\lambda)y) = 0$ then $\lambda\alpha(x) = (1-\lambda)\alpha(y)$ and also $0 = \sum \lambda\alpha_i(x) - (1- \lambda)\alpha_i(y) = \lambda \sum \alpha_i(x) - (1-\lambda)\sum \alpha_i(y) = 2\lambda-1$, which implies $\lambda = \frac12$ and thus $\alpha(x) = \alpha(y)$. If $F(\lambda x + (1-\lambda)y) \in D \setminus \{0\}$, then $\alpha_i(x) > 0$ for all $i \in [d+1]$ or $\alpha_i(y) > 0$ for all $i \in [d+1]$ depending on whether the coordinates of $F(\lambda x + (1-\lambda)y)$ are positive or negative. Thus either $x \in \bigcap A_i$ or $y \in \bigcap A_i$.
\end{proof} 

Theorem~\ref{thm:RadonKKM} easily implies both the topological Radon theorem and the KKM theorem. To derive the topological Radon theorem as a consequence, note that for a map $\alpha\colon \Delta_d \to \Delta_{d-1} \subset \Delta_d$, the coordinate functions $\alpha_1, \dots, \alpha_{d+1}$ are subordinate to the open cover $A_i = \{x \in \Delta_d \ : \ \alpha_i(x) > 0\}$, where $A_{d+1} = \emptyset$. Thus $\bigcap A_i = \emptyset$ and $\alpha$ has a Radon partition by Theorem~\ref{thm:RadonKKM}. 

We can regard Theorem~\ref{thm:RadonKKM} as a natural strengthening of the KKM theorem, where the KKM condition (that $\Delta_d^J \subseteq \bigcup_{j \in J} A_j$ for all $J \subseteq [d+1]$) is replaced by the condition that a partition of unity subordinate to the cover avoids a Radon partition, which is a weaker requirement. We can now prove a colorful generalization:

\begin{thm}
\label{thm:colKKM}
    Let $\alpha^{(1)}, \dots, \alpha^{(d+1)} \colon \Delta_d \to \Delta_d$ be continuous maps such that for every $J \subseteq [d+1]$ and for every $i \in [d+1]$ we have that $\alpha^{(i)}(\Delta_d^J) \subseteq \Delta_d^J$. Then there is an $x \in \Delta_d$ and a permutation $\pi$ of $[d+1]$ such that $\alpha^{(i)}_{\pi(i)}(x) \ge \alpha^{(i)}_j(x)$ for all $j \in [d+1]$.
\end{thm}

\begin{proof}
    Let $A \colon \Delta_d \to \R^{(d+1) \times (d+1)}, \ x \mapsto (\alpha^{(i)}_j(x))_{i,j}$. The condition $\alpha^{(i)}(\Delta_d^J) \subseteq \Delta_d^J$ implies that for $x \in \Delta_d^J$ the matrix $A(x)$ has non-zero entries only in the columns corresponding to $j \in J$. Let
    $$F\colon (\Delta_d)^{*2}_\Delta \to \R^{(d+1) \times (d+1)}, \ \lambda x + (1-\lambda)y \mapsto \lambda A(x) - (1-\lambda)A(y).$$
    We observe that no column of~${F(\lambda x + (1-\lambda)y)}$ has both positive and negative entries: Indeed, if $\lambda=1$ all entries of ${F(\lambda x + (1-\lambda)y)}$ are non-negative; if $\lambda= 0$ all entries of ${F(\lambda x + (1-\lambda)y)}$ are non-positive; and if $0 < \lambda < 1$ then $x$ and $y$ are in proper faces of~$\Delta_d$, which are disjoint by definition of deleted join, say $x \in \Delta_d^J$ and $y \in \Delta_d^{[d+1] \setminus J}$. Then since $\alpha^{(i)}(\Delta_d^J) \subseteq \Delta_d^J$, columns of ${F(\lambda x + (1-\lambda)y)}$ corresponding to $j \in J$ are non-negative and all other columns are non-positive. Since each column of $F(\lambda x + (1-\lambda)y)$ is either entirely non-negative or entirely non-positive, $F(\lambda x + (1-\lambda)y)$ has rows in intersecting cube facets. 
    
    By Theorem~\ref{thm:colorfulBU} there is a point $\lambda x + (1-\lambda)y \in (\Delta_d)^{*2}_\Delta$ and a permutation $\pi$ of $[d+1]$ such that $F(\lambda x + (1-\lambda)y)_{\pi(i)i} = \lambda\alpha^{\pi(i)}_i(x) - (1-\lambda)\alpha^{\pi(i)}_i(y)$ is non-negative and $|F(\lambda x + (1-\lambda)y)_{\pi(i)i}| \ge |F(\lambda x + (1-\lambda)y)_{\pi(i)j}|$ for all~$i, j \in [d+1]$. If some $F(\lambda x + (1-\lambda)y)_{\pi(i)i}$ were zero, then since these entries maximize their respective rows, the entire row would be zero, which is impossible. Thus all $F(\lambda x + (1-\lambda)y)_{\pi(i)i}$ are positive. In particular, since the sign of columns is constant, all columns are non-negative. Notice that it is also not possible for an entire column to be zero, since one entry in this column maximizes its row in absolute value and this row is not identically zero. By the above this can only be the case when $\lambda = 1$. Thus $F(\lambda x + (1-\lambda)y)_{\pi(i)i} = \alpha^{\pi(i)}_i(x) > 0$, and $|F(\lambda x + (1-\lambda)y)_{\pi(i)i}| \ge |F(\lambda x + (1-\lambda)y)_{\pi(i)j}|$ for all~$i, j \in [d+1]$ gives that $\alpha^{\pi(i)}_i(x) \ge \alpha^{\pi(i)}_j(x)$ for all~$i, j \in [d+1]$.
\end{proof}

We derive the colorful KKM theorem, originally due to Gale~\cite{Gale1984}, as an immediate consequence of Theorem~\ref{thm:colKKM}:

\begin{cor}[colorful KKM theorem]
\label{cor:colKKM}
    For each $i \in [d+1]$ let $A_1^{(i)}, \dots, A_{d+1}^{(i)}$ be open covers of~$\Delta_d$ such that for every $J \subseteq [d+1]$ and for every $i \in [d+1]$ we have that $\Delta_d^J \subseteq \bigcup_{j \in J} A_j^{(i)}$. Then there is a permutation $\pi$ of $[d+1]$ such that $\bigcap A^{(i)}_{\pi(i)} \ne \emptyset$.
\end{cor}

\begin{proof}
    Let $\alpha^{(i)} \colon \Delta_d \to \Delta_d$ be a partition of unity subordinate to the cover $\{A_1^{(i)}, \dots, A_{d+1}^{(i)}\}$. Then for every $J \subseteq [d+1]$ and for every $i \in [d+1]$ we have that $\alpha^{(i)}(\Delta_d^J) \subseteq \Delta_d^J$, since $\Delta_d^J \subseteq \bigcup_{j \in J} A_j^{(i)}$. Thus by Theorem~\ref{thm:colKKM} there is an $x \in \Delta_d$ and a permutation $\pi$ of $[d+1]$ such that $\alpha^{(i)}_{\pi(i)}(x) \ge \alpha^{(i)}_j(x)$ for all $j \in [d+1]$. In particular, $\alpha^{(i)}_{\pi(i)}(x) > 0$ and thus $x \in \bigcap A^{(i)}_{\pi(i)}$.
\end{proof}

\begin{rem}
    In the same way that the Borsuk--Ulam theorem strengthens Brouwer's fixed point theorem, Theorem~\ref{thm:colKKM} and Corollary~\ref{cor:colKKM} exhibit the colorful Borsuk--Ulam theorem as a strengthening of Gale's colorful KKM theorem. Here it is interesting that in the proof of Theorem~\ref{thm:colKKM} we used that the columns of any matrix in the image of the map~$F$ are either non-negative or non-positive. This is a much stronger condition than is necessary for the application of Theorem~\ref{thm:colorfulBU}. 
\end{rem}

As a consequence of Corollary~\ref{cor:colKKM} we can state a colorful Brouwer's fixed point theorem that for $d+1$ maps $f_i \colon \Delta_d \to \Delta_d$ asserts the existence of a point $x\in \Delta_d$ and a set of inequalities that in the case $f_1 = f_2 = \dots = f_{d+1}$ specialize to $x$ is a fixed point. We introduce one piece of terminology: Let $S(d) \subseteq \R^{d \times d}$ be the set of \emph{stochastic matrices}, that is, $A \in S(d)$ if all entries of $A$ are non-negative and every row sums to one. Thus stochastic matrices are those $(d \times d)$-matrices, where every row vector belongs to the $(d-1)$-simplex~$\Delta_{d-1}$.

\begin{thm}
\label{thm:colBrouwer}
    Let $f\colon \Delta_d \to S(d)$ be continuous. Then there is an $x \in \Delta_d$ and a permutation $\pi$ of $[d+1]$ such that $f_{i\pi(i)}(x) \le x_{\pi(i)}$ for all $i\in[d+1]$.
\end{thm}

\begin{proof}
    Let $A_j^{(i)} = \{x \in \Delta_d \ : \ f_{ij}(x) \le x_j\}$. These sets are closed by continuity of~$f$. Let $J \subseteq [d+1]$ be some non-empty set. Then for $x \in \Delta_d^J$ we have that $\sum_{j \in J} x_j = 1$ and $\sum_{j \in J} f_{ij}(x) \le 1$ for every $i \in [d+1]$. This implies that for some $j \in J$ we have that $f_{ij}(x) \le x_j$ and thus $x \in A^{(i)}_j$. This shows that $\Delta_d^J \subseteq \bigcup_{j \in J} A^{(i)}_j$. A standard approximation argument shows that Corollary~\ref{cor:colKKM} also holds for collections of closed sets, and thus we get that there is a permutation $\pi$ of $[d+1]$ such that $\bigcap A^{(i)}_{\pi(i)} \ne \emptyset$. Any $x \in \bigcap A^{(i)}_{\pi(i)}$ satisfies the desired set of inequalities.
\end{proof}

This is indeed a colorful generalization of Brouwer's fixed point theorem. If every row of $f\colon \Delta_d \to S(d)$ is equal to $h \colon \Delta_d \to \Delta_d$, then Theorem~\ref{thm:colBrouwer} asserts the existence of $x\in \Delta_d$ with $h_i(x) \le x_i$ for all $i \in [d+1]$. Since $\sum_i h_i(x) = 1 = \sum_i x_i$, this implies $h(x) = x$.

\section{The colorful ham sandwich theorem}

We first recall the classical Ham Sandwich theorem, conjectured by Steinhaus and proved by Banach; see~\cite{beyer2004}. Here \emph{hyperplane} refers to an affine subspace of codimension one. We will think of every hyperplane $H \subseteq \R^d$ as coming with a fixed orientation so that its positive halfspace~$H^+$ is well-defined (similarly, its negative halfspace~$H^-$), that is, for $H = \{x \in \R^d \ : \ \langle x, z \rangle = b\}$ for $z \in \R^d \setminus \{0\}$ and $b \ge 0$, we let $H^+ = \{x \in \R^d \ : \ \langle x,z \rangle \ge b\}$ and $H^- = \{x \in \R^d 
\ : \ \langle x,z \rangle \le b\}$. 

\begin{thm}[Ham Sandwich theorem]
\label{thm:hs}
    Let $\mu_1, \mu_2,\dots \mu_{d}$ be Borel probability measures on~$\R^d$ such that for every hyperplane $H$ we have that $\mu_i(H) = 0$ for all $i \in [d]$. Then there is a hyperplane~$H$ such that $\mu_i(H^+)=\frac{1}{2} = \mu_i(H^-)$ for all $i\in[d]$.
\end{thm}

A colorful version of Theorem~\ref{thm:hs} will take several families of Borel probability measures as input and guarantee the existence of a hyperplane that separates the measures in a certain way. To be a true colorful version, such a result should specialize to Theorem~\ref{thm:hs} if all families of Borel measures are the same. We will discuss two natural attempts at formulating a colorful Ham Sandwich theorem, for simplicity for $d=2$.

Given two families of Borel probability measures $\sigma_1=\set{r_1,g_1},\sigma_2=\set{r_2,g_2}$ in~$\R^2$, is there a line~$H$ such that $r_1(H^+)\ge \frac{1}{2}, r_2(H^-)\ge \frac{1}{2}$ and $g_1(H^-)\ge \frac{1}{2}, g_2(H^+)\ge \frac{1}{2}$? For  $r_1 = r_2$ and $g_1 = g_2$ this would reduce to the usual Ham Sandwich theorem. However, this proposed colorful version is false: Figure~\ref{fig:badcolHS} shows two families of two Borel measures (red and green) distributed along the parabola such that no line $H$ as above exists.

\begin{figure}[ht]
\centering
\begin{tikzpicture}[xscale=0.3, yscale=0.3]
\draw[line width= 4pt, red, opacity=0.5, scale =0.15, domain = -5:0, variable = \x]  plot ({5*\x},{2*\x*\x+1});
 \draw[line width= 4pt, green, opacity=0.5, scale =0.15, domain = 0:5, variable = \x]  plot ({5*\x},{2*\x*\x+1});

\draw (-5,4)--(5,4);
\draw[->, gray] (-5,0)--(5,0);
\draw[->, gray] (0,-2)--(0,10);
\end{tikzpicture}
\hspace{0.1in}
\begin{tikzpicture}[xscale=1.7, yscale=1.7]

\fill[fill=red, fill opacity=0.5](0.1, 1.5)--(2, 1.5)--(2, 1.2)--(0.1,1.2)--cycle;
\fill[fill=green, fill opacity=0.5](2, 1.5)--(3.9, 1.5)--(3.9, 1.2)--(2,1.2)--cycle;
\fill[fill=red, fill opacity=0.5](0.9, 0.7)--(2.9, 0.7)--(2.9, 1)--(0.9,1)--cycle;
\fill[fill=green, fill opacity=0.5](0.1, 0.7)--(0.9, 0.7)--(0.9, 1)--(0.1,1)--cycle;
\fill[fill=green, fill opacity=0.5](2.9, 0.7)--(3.9, 0.7)--(3.9, 1)--(2.9,1)--cycle;

\draw (1,0.5) -- (1,2);
\draw (3,0.5) -- (3,2);
\draw[->] (1,1.8)--(0.1,1.8);
\draw[->] (3,1.8)--(2.1,1.8);

\node (A) at (0.5,1.35){$r_1$};
\node (B) at (2.5,1.35){$g_1$};
\node (C) at (1.5,0.85){$r_2$};
\node (D) at (0.5,0.85){$g_2$};

\end{tikzpicture}
\caption{We concentrate two families of measures along the parabola in the plane. Every line $H$ intersects the parabola in at most two points. If at least half the measure of $r_1$ is supposed to be between the intersection points of $H$ with the parabola, and at least half the measure of $g_1$ is outside the intersection points, then the intersections of $H$ with the parabola are constrained as indicated above. If then $r_2$ and $g_2$ are distributed along the parabola as indicated, either more than half of $r_2$ is between the intersection points, or more than half of $g_2$ is outside the intersection points.}
\label{fig:badcolHS}
\end{figure}
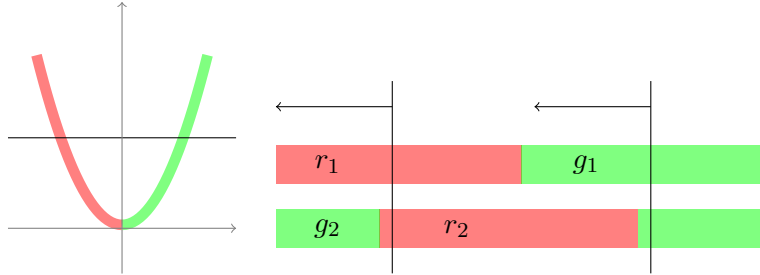
   
Since this proposed colorful generalization of the Ham Sandwich theorem has a conclusion that is too strong to be true, a refined attempt at arriving at a colorful version might allow us to switch the roles of $g_2$ and $r_2$ in the second family of Borel measures. That is, given two families of Borel probability measures $\sigma_1=\set{r_1,g_1}, \sigma_2=\set{r_2,g_2}$, is there a line~$H$ such that $r_1(H^+)\ge \frac{1}{2}, r_2(H^-)\ge \frac{1}{2}$ and either $g_1(H^+)\ge \frac{1}{2}, g_2(H^-)\ge \frac{1}{2}$ or $g_1(H^-)\ge \frac{1}{2}, g_2(H^+)\ge \frac{1}{2}$? Again, if $r_1 = r_2$ and $g_1 = g_2$ then this reduces to the usual Ham Sandwich theorem. While this colorful version of the Ham Sandwich theorem is true, it is a trivial consequence of the Ham Sandwich theorem itself: Let $H$ be a line that simultaneously bisects the measures $r_1$ and $g_1$. Then $r_1(H^+) = \frac12 = r_1(H^-)$. By flipping the orientation of $H$ if necessary, we can make sure that $r_2(H^-) \ge \frac12$. Moreover, $g_1(H^+) = \frac12 = g_1(H^-)$ and one of $g_2(H^+) \ge \frac12$ or $g_2(H^-) \ge \frac12$ has to hold as well.

Nevertheless, the Ham Sandwich theorem admits a (non-trivial) colorful generalization. The example above shows that we need to impose that measures in different families are not ``oppositely distributed.'' We make this notion precise now and then state the colorful ham sandwich theorem.

Let $\mathcal M = \{\mu_1, \dots, \mu_m\}$ be a family of finite Borel measures on~$\R^d$, and let $H$ be a hyperplane. We say that $\mu_i$ \emph{maximizes $H^+$ for $\mathcal M$} if $\mu_i(H^+) -\frac12\mu_i(\R^d) \ge \mu_\alpha(H^+) - \frac12\mu_\alpha(\R^d)$ for all $\alpha \in [m]$. Similarly, $\mu_i$ \emph{minimizes $H^+$ for $\mathcal M$} if $\mu_i(H^+)-\frac12\mu_i(\R^d) \le \mu_\alpha(H^+) - \frac12\mu_\alpha(\R^d)$ for all $\alpha \in [m]$. We may now state the colorful generalization of the Ham Sandwich theorem:

\begin{thm}
\label{colHS}
    For each $j \in [d+1]$ let $\mathcal M_j = \{\mu_1^{(j)}, \dots, \mu_{d+1}^{(j)}\}$ be an ordered family of $d+1$ finite Borel measures on~$\R^d$  such that for every hyperplane $H$ we have that $\mu_i^{(j)}(H) = 0$ for all $i \in [d+1]$. Suppose further that for $j, k \in [d+1]$ with $j \ne k$ and any hyperplane $H$ we have that if $\mu_i^{(j)}$ maximizes $H^+$ for $\mathcal M_j$ then $\mu_i^{(k)}$ does not minimize $H^+$ for~$\mathcal M_k$.
    Then there is permutation $\pi$ of $[d+1]$ such that $\mu_{i}^{(\pi(i))}(H^{+})-\frac12\mu_i^{(\pi(i))}(\R^d) \ge \mu_j^{(\pi(i))}(H^+) - \frac12\mu_j^{(\pi(i))}(\R^d)$ for all~$i,j \in [d+1]$. 
\end{thm}

\begin{proof}
Let $u = (u_0,u_1, \dots, u_d) \in S^d$. If there is an $i\in[d]$ such that $u_i\neq 0$, i.e., if $u_0 \ne \pm 1$, assign to $u$ the halfspace 
\[H^+(u) = \{(x_1,\dots, x_d)\in \R^d \ | \ u_1x_1+\dots u_dx_d\le u_0\}\]

Notice that
\[H^+(-u) = \{(x_1,\dots, x_d)\in \R^d \ | \ -u_1x_1 -\dots -u_dx_d\le -u_0\}\]
\[=\{(x_1,\dots, x_d)\in \R^d \ | \ u_1x_1+\dots +u_dx_d\ge u_0\}=H^-(u).\]

Define $f_{ji}(u)=\mu_i^{(j)}(H^+(u)) - \frac12\mu_i^{(j)}(\R^d)$. 
These maps are continuous and admit a continuous extension into the north and south pole. Thus the maps $f_{ij}$ define an odd and continuous map $F \colon S^d \to \R^{(d+1) \times (d+1)}$. If $F(u)$ does not have rows in intersecting cube facets then (writing $H = H(u)$) there are $i,j,k \in [d+1]$ with $j \ne k$ such that 
\begin{align*}
    &|\mu_i^{(j)}(H^+) - \tfrac12\mu_i^{(j)}(\R^d)| \ge |\mu_\alpha^{(j)}(H^+) - \tfrac12\mu_\alpha^{(j)}(\R^d)| \ \text{and} \\ 
    &|\mu_i^{(k)}(H^+) - \tfrac12\mu_i^{(k)}(\R^d)| \ge |\mu_\alpha^{(k)}(H^+) - \tfrac12\mu_\alpha^{(k)}(\R^d)| \ \text{for all} \ \alpha \in [d+1],
\end{align*}
where $\mu_i^{(j)}(H^+) - \frac12\mu_i^{(j)}(\R^d) \ge 0$ and $\mu_i^{(k)}(H^+) - \frac12\mu_i^{(k)}(\R^d) \le 0$. 
This means that $\mu_i^{(j)}$ maximizes $H^+$ for $\mathcal M_j$ and $\mu_i^{(k)}$ minimizes $H^+$ for~$\mathcal M_k$, in contradiction to our assumption. Thus $F(u)$ has rows in intersecting cube facets for every $u \in S^d$. Applying Theorem~\ref{thm:colorfulBU} finishes the proof.
\end{proof}

The uncolored version of Theorem~\ref{colHS}, that is, when $\mu_i^{(1)} = \mu_i^{(2)} = \dots = \mu_i^{(d+1)}$, is the following strengthening of the Ham Sandwich theorem:

\begin{thm}
\label{kfhs}
    Let $\mu_1, \mu_2,\dots \mu_{d+1}$ be finite Borel measures on~$\R^d$ such that for every hyperplane~$H$ we have that $\mu_i(H) = 0$ for all $i \in [d+1]$. Then there is a hyperplane $H$ such that $\mu_i(H^+)- \mu_i(H^-)=\mu_j(H^+) - \mu_j(H^-)$ for all $i,j\in[d+1]$.
\end{thm}

\begin{proof}
    Apply Theorem~\ref{colHS} in the case $\mathcal M_1 = \mathcal M_2 = \dots = \mathcal M_{d+1} = \{\mu_1, \dots, \mu_{d+1}\}$. First suppose that there exists some hyperplane $H$ such that the halfspace $H^+$ both maximizes and minimizes one of the~$\mu_i$. This implies $\mu_i(H^+) -\frac12\mu_i(\R^d) = \mu_\alpha(H^+) - \frac12\mu_\alpha(\R^d)$ for all $\alpha \in [d+1]$. Since $\mu_i(H^+) + \mu_i(H^-) = \mu_i(\R^d)$, this implies $\mu_i(H^+)- \mu_i(H^-)=\mu_\alpha(H^+) - \mu_\alpha(H^-)$ for all $\alpha \in[d+1]$. 
    
    If no $\mu_i$ simultaneously maximizes and minimizes some halfspace~$H^+$, then by Theorem~\ref{colHS}, we again get that $\mu_i(H^+) -\frac12\mu_i(\R^d) = \mu_\alpha(H^+) - \frac12\mu_\alpha(\R^d)$ for all $\alpha \in [d+1]$.
\end{proof}

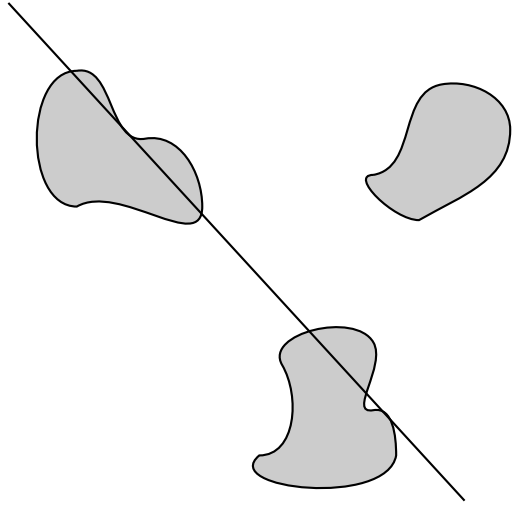
\begin{figure}
\centering
\begin{tikzpicture}[scale=0.6]

    \filldraw[thick, fill=black!20!white](6,-1.5)
    to [out=-60,in=190] (6,-3)
    to [out=10,in=90] (6.5,-4)
    to [out=-100,in=-140] (3.5,-4)
    to [out=0,in=-60] (4,-2)
    to [out=120, in=120] (6,-1.5);

    \filldraw[thick, fill=black!20!white](-0.5,4.5)
    to [out=10,in=190] (1,3)
    to [out=10,in=90] (2.25,1.5) 
    to [out=-90,in=30] (-0.5,1.5)
    to [out=180, in=180] (-0.5,4.5);

    \filldraw[thick, fill=black!20!white](6,2.2)
    to [out=10,in=190] (7.5,4.2)
    to [out=10,in=90] (9,3.2) 
    to [out=-90,in=30] (7,1.2)
    to [out=180, in=180] (6,2.2); 
    
    \draw[thick] (-2,6) -- (8,-5);
    
\end{tikzpicture}
\caption{An example of (supports of) three well-separated measures in the plane. If the supports are bounded then no line can intersect all three measures.}
\label{fig:sepmass}
\end{figure}

A family $\mathcal F$ of convex sets in $\R^d$ is \emph{well-separated} if any collection $x_1, \dots, x_k$ of points from pairwise distinct $K_1, \dots, K_k \in \mathcal F$ is in general position, that is, for any $k \le d+1$, any pairwise distinct $K_1, \dots, K_k \in \mathcal F$ and any $x_1 \in K_1, \dots, x_k \in K_k$ the set $\{x_1, \dots, x_k\}$ is not contained in a common $(k-2)$-dimensional affine subspace. See Figure~\ref{fig:sepmass} for an example. We call a family $\mathcal M$ of Borel measures on $\R^d$ \emph{well-separated} if the family of convex hulls of supports $\mathcal F = \{\conv(\mathrm{supp} \ \mu) \ : \ \mu \in \mathcal M\}$ is well-separated. The \emph{support} of a Borel measure $\mu$ on $\R^d$ is $\mathrm{supp} \ \mu = \{x \in \R^d \ : \ \forall \varepsilon > 0 \ \mu(B_\varepsilon(x)) > 0\}$. In particular, if $\mathcal F$ is a well-separated family of $d+1$ sets in~$\R^d$, then no hyperplane can intersect all sets in~$\mathcal F$. 

B\'{a}r\'{a}ny, Hubard, and Jer\'{o}nimo~\cite{BaranyHubardJeronimo2008} show that if $\mu_1, \dots, \mu_d$ are finite Borel measures on $\R^d$ with bounded supports that are well-separated and such that $\mu_i(H) = 0$ for every hyperplane $H$ in~$\R^d$, then there is a hyperplane that cuts off a specified fraction from each measure~$\mu_i$, that is, for $\alpha_1,\dots,\alpha_d\in (0,1)$ there is a hyperplane $H$ such that  $\mu_i(H^+)=\alpha_i\mu_i(\R^d)$ for all $i\in[d]$.
We use Theorem~\ref{kfhs} to show the following variant:

\begin{cor}[Variant of a result of B\'{a}r\'{a}ny, Hubard, Jer\'{o}nimo~\cite{BaranyHubardJeronimo2008}]
\label{cor:bhj}
Let $\mu_1, \mu_2,\dots \mu_d$ be finite Borel measures on~$\R^d$ with bounded supports and $\mu_i(H) = 0$ for all hyperplanes $H$ in $\R^d$ and for all $i \in [d]$. Suppose there is an $x \in \R^d$ such that no hyperplane through~$x$ may intersect the supports of all~$\mu_i$. Then for all $\alpha_1,\dots,\alpha_d\in (0,1)$ there is a hyperplane $H$ such that $\mu_i(H^+)=\alpha_i\mu_i(\R^d)$ for all $i\in[d]$.
\end{cor}

\begin{proof}
    Normalize the measures $\mu_i$ such that $\alpha_i\mu_i(\R^d) = 1$ for all~$i$ by dividing each $\mu_i$ by $\alpha_i\mu_i(\R^d)$. In particular, after this normalization $\mu_i(\R^d) = \frac{1}{\alpha_i} > 1$ for all~$i$. 
    By a standard compactness argument there is an $\varepsilon > 0$ such that any hyperplane that intersects $B_\varepsilon(x)$ does not intersect the supports of all~$\mu_i$. Thus we may construct a Borel measure $\mu_{d+1}$ on $\R^d$ supported in $B_\varepsilon(x)$ with continuous density and with $\mu_{d+1}(\R^d) = 1$. Now apply Theorem~\ref{kfhs} to this collection, which yields a hyperplane $H$ with $\mu_i(H^+)- \mu_i(H^-)=\mu_j(H^+) - \mu_j(H^-)$ for all $i,j\in[d+1]$. The hyperplane $H$ cannot intersect the supports of all measures $\mu_1, \dots, \mu_{d+1}$. Let $i\in [d+1]$ such that $H$ is disjoint from the support of~$\mu_i$. If $i \ne d+1$ then $\mu_i(H^+)-\mu_i(H^-) = \mu_i(\R^d) > 1$, but $|\mu_{d+1}(H^+)-\mu_{d+1}(H^-)| \le 1$, so $\mu_i(H^+)- \mu_i(H^-)\ne \mu_{d+1}(H^+) - \mu_{d+1}(H^-)$. Thus $i = d+1$. This implies $\mu_j(H^+) - \mu_j(H^-) = \mu_{d+1}(H^+) - \mu_{d+1}(H^-) =1$, which finishes the proof.
\end{proof}

Notice that Corollary~\ref{cor:bhj} contains the result of B\'{a}r\'{a}ny, Hubard, Jer\'{o}nimo as a special case, provided that for any family $\{K_1, \dots, K_d\}$ of compact, convex sets in~$\R^d$ that are well-separated, there is an $x\in \R^d$ such that $\{K_1,\dots, K_d, \{x\}\}$ is well-separated. We have been unable to show this.

\section{Colorful Borsuk--Ulam theorems for higher symmetry}

The methods we have used here to prove colorful results for symmetric set coverings of spheres generalize easily to other settings. Here we prove generalizations for free $\Z/p$-actions on spheres, $p$ a prime. Below $s\in\Z/p$ acts on $(j,t) \in [d] \times \Z/p$ by $s\cdot (j,t) = (j, t+s)$.

\begin{thm}
\label{thm:p-fold-kyfan}
    Let $p$ be a prime. Let $d\ge1$ and $n = (p-1)d-1$ be integers. Fix some free $\Z/p$-action on~$S^n$. Let $\Sigma$ be a $\Z/p$-equivariant triangulation of $S^n$ with vertex set~$V$. Let $\ell\colon V \to [d] \times \Z/p$ be $\Z/p$-equivariant. Fix $s_1, \dots, s_d \in \Z/p$. Then either there is a $(p-1)$-face $\sigma$ of $\Sigma$ with $\ell(\sigma) = \{j\} \times \Z/p$ for some $j \in [d]$ or there is a facet $\sigma$ of $\Sigma$ with $\ell(\sigma) = \{(j, s) \ : \ j \in [d], \ s \in \Z/p \setminus \{s_j\}\}$.
\end{thm}

\begin{proof}
    If no $(p-1)$-face is labelled with all elements in $\{j\}\times\Z/p$ then $\ell$ induces an equivariant map $\Sigma \to (\partial\Delta_{p-1})^{\ast d}$ by identifying each label $(j,s)$ with the vertex $s$ in the $j$th copy $\partial \Delta_{p-1}$. The $d$-fold join~$(\partial\Delta_{p-1})^{\ast d}$ is a sphere of dimension~$n$. By Lemma~\ref{lem:dold} such an equivariant map will have non-zero degree and thus be surjective. In particular, some face $\sigma$ of $\Sigma$ maps to the facet $\{(j, s) \ : \ j \in [d], \ s \in \Z/p \setminus \{s_j\}\}$ of~$(\partial\Delta_{p-1})^{\ast d}$.
\end{proof}

\begin{thm}
\label{thm:p-cover}
    Let $p$ be a prime. Let $d\ge1$ and $n = (p-1)d-1$ be integers. 
    Let $A_i \subset S^n$ be closed sets for $i \in [d]$ such that $S^n = \bigcup_i (A_i \cup s\cdot A_i \cup \dots \cup s^{p-1}A_i)$. Suppose that $\bigcap_k \varepsilon^k\cdot A_i = \emptyset$ for every $i \in [d]$.
    Then for all $s_1, \dots, s_d \in \Z/p$ we have that $\bigcap_i \bigcap_{s \ne s_i} s\cdot A_i \ne \emptyset$.
\end{thm}

\begin{proof}
    Assume that $\bigcap_k s^k\cdot A_i = \emptyset$ for all~$i$. Let $T_\varepsilon$ be a $\Z/p$-symmetric triangulation of~$S^d$ such that each facet has diameter less than~$\varepsilon$, where $\varepsilon >0$ is chosen such that any set of diameter less than $\varepsilon$ intersects at most $p-1$ of the sets $A_i, s\cdot A_i, \dots, s^{p-1}A_i$. This can be achieved by taking repeated barycentric subdivisions of a given $\Z/p$-symmetric triangulation. 

    Let $\ell\colon V(T_{\varepsilon})\to [d] \times \Z/p$ be a labelling of the vertices of~$T_\varepsilon$ such that $\ell(v) = (i,g)$ only if $v \in g\cdot A_i$. By our choice of~$\varepsilon$, there is no face $\sigma$ with $\ell(\sigma) = \{j\} \times \Z/p$. By Theorem~\ref{thm:p-fold-kyfan} there is a facet labelled precisely by the set $\{(i, s) \ : \ i \in [d], \ s \in \Z/p \setminus \{s_i\}\}$. Let $x_\varepsilon$ be the barycenter of some such facet. As $\varepsilon$ approaches zero, by compactness of~$S^d$, the $x_\varepsilon$ have an accumulation point~$x$. Since the $A_i$ are closed, we have that $\bigcap_i \bigcap_{s\ne s_i} s\cdot A_i$.
\end{proof}

\begin{rem}
\label{rem:n++}
    The proof of Theorem~\ref{thm:p-fold-kyfan} shows that if in Theorem~\ref{thm:p-fold-kyfan} we increase $n$ by one, that is, $n = (p-1)d$, then the first alternative will always occur: There is a face $\sigma$ of $\Sigma$ labelled with an entire $\Z/p$-orbit, that is, $\ell(\sigma) = \{j\} \times \Z/p$ for some $j \in [d]$.

    Similarly, for Theorem~\ref{thm:p-cover}, if $n = (p-1)d$ then it is impossible that $\bigcap_k s^k\cdot A_i = \emptyset$ for every $i \in [d]$.
\end{rem}

\begin{cor}
\label{cor:kyfan-dold}
    Let $p$ be a prime. Let $d \ge 1$ and $n = (p-1)d-1$ be integers. Let $f \colon S^n \to \R^d$ be continuous. Then there is a $\Z/p$-orbit $x, s\cdot x, \dots, s^{p-1}\cdot x$ such that $f$ maps $p-1$ points in this orbit to the same point $y$ in~$\R^d$ and the remaining point to $y-(\alpha, \dots, \alpha)$ for some $\alpha \in \R$.
\end{cor}

\begin{proof}
    For $x \in S^n$ denote its $\Z/p$-orbit $\{x, s\cdot x, \dots, s^{p-1}\cdot x\}$ by $G\cdot x$. Denote the $i$th coordinate function of $f\colon S^n \to \R^d$ by $f_i \colon S^n \to \R$. Let $x \in S^n$. We place $x \in A_i$ if $\diam(f_i(G\cdot x)) \ge \diam(f_j(G\cdot x))$ for all $j \in [d]$ and $f_i(x) \ge f_i(s^k \cdot x)$ for all $k \in [p]$. Thus $x \in A_i$ if $f_i$ fluctuates at least as much on the orbit of $x$ as any other coordinate function~$f_j$, and additionally $f_i(x)$ is the largest value in the orbit of~$x$. As both of these values have to maximized somewhere $S^n = \bigcup_i (A_i \cup s\cdot A_i \cup \dots \cup s^{p-1}A_i)$.

    Suppose there is some point $x$ in $\bigcap_k s^k\cdot A_i$. Then $f_i(x) = f_i(s \cdot x) = \dots = f_i(s^{p-1}\cdot x)$ and $0 = \diam(f_i(G\cdot x)) \ge \diam(f_j(G\cdot x))$ for all $j \in [d]$. This implies $f(x) = f(s\cdot x) = \dots = f(s^{p-1}x)$. Otherwise by Theorem~\ref{thm:p-cover} there is some $x$ in $\bigcap_i \bigcap_{k \ne p-1} s^k\cdot A_i$. Then $f(x) = f(s\cdot x) = \dots = f(s^{p-2} \cdot x)$ and $\diam(f_i(G\cdot x)) = \diam(f_j(G\cdot x))$ for any $i,j \in [d]$. Since the first $p-1$ points in $G\cdot x$ are mapped to the same point, we have that $f_i(s^{p-1}\cdot x) = f_i(x) - \diam(f_i(G\cdot x))$ for all $i \in [d]$.
\end{proof}

\begin{rem}
    For $p=2$ and $f\colon S^{d-1} \to \R^d$ an odd map, Corollary~\ref{cor:kyfan-dold} asserts that $f$ maps a pair of antipodal points to $f(x) = y$ and $f(-x) = y - (\alpha, \dots, \alpha)$. Since $f$ is odd, $f(-x) = -y$ and thus $(\alpha, \dots, \alpha) = 2y$, that is, the corollary asserts that $f$ maps a pair of antipodal points to the $1$-dimensional diagonal in~$\R^d$.

    Following Remark~\ref{rem:n++} the proof of Corollary~\ref{cor:kyfan-dold} shows that for $n \ge (p-1)d$, we get that any continuous $f\colon S^n \to \R^d$ maps an entire $\Z/p$-orbit to the same point. This is a classical result of Bourgin--Yang and others \cite{Dold1983,yang1954,yang1955,Bourgin1955}. Corollary~\ref{cor:kyfan-dold} extends this orbit collapsing result in the same fashion that Ky Fan's theorem extends the Borsuk--Ulam theorem.
\end{rem}

We can now derive a colorful generalization of Theorem~\ref{thm:p-cover} in the same way that we showed the colorful generalization (Theorem~\ref{colkf}) Fan's theorem.

\begin{thm}
\label{thm:col-p-kyfan}
    Let $p$ be a prime. Let $d\ge1$ and $n = (p-1)d-1$ be integers. 
    Let $A_i^{(j)} \subset S^n$ be closed sets with $i \in [d]$ and $j \in [n+1]$ such that $S^n = \bigcup_i (A_i^{(j)}\cup s\cdot A_i^{(j)} \cup \dots \cup s^{p-1}\cdot A_i^{(j)})$ for every~$j \in [n+1]$. Suppose that $\bigcap_k s^k\cdot A_i^{(j_k)} = \emptyset$ for every $i \in [d]$ and for pairwise distinct $j_1, \dots, j_p \in [n+1]$.
    Then for all $s_1, \dots, s_d \in \Z/p$ we have that $\bigcap_i \bigcap_{s \ne s_i} s\cdot A_i^{(\pi(i, s))} \ne \emptyset$ for some bijection $\pi\colon \{(i,s) \in [d] \times \Z/p \ : \ s \ne s_i\} \to [n+1]$.
\end{thm}

\begin{proof}
    Let $T_\varepsilon$ be a triangulation, where every face has diameter at most~$\varepsilon > 0$. Here $\varepsilon$ is chosen sufficiently small so that no face intersects $s\cdot A_i^{(j_1)}, s^2\cdot A_i^{(j_2)}, \dots, s^p\cdot A_i^{(j_p)}$ for pairwise distinct $j_1, \dots, j_p \in [n+1]$ and any~$i$. Let $T'_\varepsilon$ denote the barycentric subdivision of~$T_\varepsilon$. Let $\ell\colon V(T'_{\varepsilon})\to [d] \times \Z/p$ be a labelling of the vertices of~$T'_\varepsilon$ such that $\ell(v) = (i,g)$ only if $v \in g\cdot A_i^{(k)}$ and $v$ subdivides a $(k-1)$-dimensional face of~$T_\varepsilon$. We may assume that $\ell$ is $\Z/p$-equivariant. 
    
    By our choice of~$\varepsilon$, there is no face $\sigma$ with $\ell(\sigma) = \{j\} \times \Z/p$. By Theorem~\ref{thm:p-fold-kyfan} there is a facet  labelled precisely by the set $\{(i, s) \ : \ i \in [d], \ s \in \Z/p \setminus \{s_i\}\}$. Let $x_\varepsilon$ be the barycenter of some such facet. As $\varepsilon$ approaches zero, by compactness of~$S^n$, the $x_\varepsilon$ have an accumulation point~$x$. For every $\varepsilon > 0$ we can find a bijection $\pi_\varepsilon\colon \{(i,s) \in [d] \times \Z/p \ : \ s \ne s_i\} \to [n+1]$ such that $x_\varepsilon$ is at distance less than $\varepsilon$ from the sets $g\cdot A_i^{(\pi(i, g))}$ with $i \in [d]$ and $g \ne g_i$. Since there are finitely many bijection $\pi\colon \{(i,s) \in [d] \times \Z/p \ : \ s \ne s_i\} \to [n+1]$, as $\varepsilon \to 0$ one such bijection will be realized infinitely many times. Call this bijection~$\pi$. Since the $A_i^{(j)}$ are closed, we have that $x\in \bigcap_i \bigcap_{s \ne s_i} s\cdot A_i^{(\pi(i, s))}$.
\end{proof}

\begin{rem}
    Theorem~\ref{thm:col-p-kyfan} has the colorful Borsuk--Ulam theorem (Theorem~\ref{thm:colorfulBU}) as a corollary: The case $p=2$ specializes to Theorem~\ref{colkf}.
\end{rem}

\end{document}